\documentclass[12pt]{amsart}

\usepackage{latexsym,amssymb,enumerate, tikz,color,mathenv}
\input{xypic}

\newcommand{\R}{\mathbb{R}}

\newcommand{\calA}{\mathcal{A}}

\newcommand{\calH}{\mathcal{H}}

\newcommand{\calF}{\mathcal{F}}

\newcommand{\bal}{\bar{\al}}

\newcommand{\bL}{\bar{L}}
\newcommand{\bR}{\bar{R}}
\newcommand{\bY}{\bar{Y}}

\newcommand{\ome}{\omega}

\newcommand{\al}{\alpha}

\newcommand{\bW}{\overline{W}}

\newcommand{\ip}[2]{\ensuremath{\langle #1,#2\rangle}}

\newcommand{\didelta}{\pi/\sqrt{\delta}}
\newcommand{\Didelta}{\frac{\pi}{\sqrt{\delta}}}

\DeclareMathOperator{\conju}{conj}
\DeclareMathOperator{\foc}{foc}
\DeclareMathOperator{\ind}{ind}

\DeclareMathOperator{\Jac}{Jac}

\DeclareMathOperator{\Sym}{Sym}

\newtheorem{thm}{Theorem}[section]
\newtheorem{lem}[thm]{Lemma}

\newtheorem{prop}[thm]{Proposition}
\newtheorem{cor}[thm]{Corollary}
\newtheorem{defi}[thm]{Definition}

\newtheorem{theorem}{Theorem}

\theoremstyle{definition}

\newenvironment{conj}{
\noindent\textsc{\textbf{Conjecture:}}
}

\begin{document}

\author{David Gonz\'alez-\'Alvaro}
\address{ Department of Mathematics, Universidad Aut\'onoma de Madrid, and ICMAT CSIC-UAM-UCM-UC3M}
\curraddr{}
\email{dav.gonzalez@uam.es}

\author{Luis Guijarro}
\address{ Department of Mathematics, Universidad Aut\'onoma de Madrid, and ICMAT CSIC-UAM-UCM-UC3M}
\curraddr{}
\email{luis.guijarro@uam.es}
\thanks{Both authors were supported by research grants  MTM2011-22612 from the Ministerio de Ciencia e Innovaci\'on (MCINN) and MINECO: ICMAT Severo Ochoa project SEV-2011-0087; the first author was also suppported by FPI grant BES-2012-053704.}

\thanks{}

\title[Restrictions on positively curved submersions]{Soft restrictions on positively curved Riemannian submersions}


\begin{abstract} 
We bound the dimension of the fiber of a Riemannian submersion from a positively curved manifold in terms of the dimension of the base of the submersion and either its conjugate radius or the length of its shortest closed geodesic.
\end{abstract}

\maketitle

\section{Introduction and statement of results}

The main difficulty when studying positively curved manifolds is the small number of known  examples. New ones appear in increasing periods of time, and at the present state of knowledge,  Riemannian submersions are necessary in their construction: starting with the correct manifold with nonnegative sectional curvature as total space, one searches for some submersion that would guarantee a positively curved base thanks to the well-known O'Neill's formula. However, this is not so easily done, pointing out to the possible presence of restrictions on the existence of such Riemannian submersions from an arbitrary nonnegatively curved manifold.

In this note, we consider the case of positively curved domains. The following conjecture (attributed to F. Wilhelm) is of particular interest: 

\medskip

\begin{conj}
Let $\pi:M^{n+k}\to B^n$ be a Riemannian submersion between compact positively curved Riemannian manifolds. Then $k\leq n-1$.
\end{conj}

\medskip

The conjecture is similar to the Chern-Kuiper theorem ruling out isometric immersions from compact nonpositively curved manifolds of dimension $n$  into nonpositively curved manifolds of dimension $2n-1$. 
 
Partial progress towards the conjecture appears in the thesis of W. Jim\'enez \cite{BG} where he used results of Kim and Tondeur \cite{KT} to obtain that if $\sec_M\geq 1$ and $\sec_B\leq C$, then 
\begin{equation}\label{Bill}
k\leq \frac{1}{3}\left(C-1\right)\left(n-1\right).
\end{equation}

It is worth noticing that O'Neill's formula together with \cite{Wals} guarantees that $C> 1$, and therefore the right hand side in \eqref{Bill} is positive.

For a different type of restrictions using rational homotopy theory methods, see \cite{AK}.

In this paper, we examine the index of Lagrangian subspaces of Jacobi fields (see section \ref{abstract jacobi} for the definitions) along horizontal geodesics to prove:

\begin{theorem}\label{conj}
Let $M^{n+k}, B^n$ be a compact Riemannian manifolds with $\sec\geq 1$, and 
$\pi:M^{n+k}\to B^n$ a Riemannian submersion. Then 
\[
k\leq \left(\frac{\pi }{\conju(B)}-1\right)(n-1),
\] 
where $\conju(B)$ is the conjugate radius of $B$.
\end{theorem}


Since the conjugate radius of a manifold with $1\leq \sec\leq C$ is at least $\pi/\sqrt{C}$, Theorem \ref{conj} gives the following improvement of Jimenez's result:

\begin{cor}\label{upper curvature bound theorem}
Under the conditions of Theorem \ref{conj}, if $\sec_B\leq C$, then
\[
k\leq (\sqrt{C}-1)\left(n-1\right)
\]
\end{cor}

This bound is better than Jimenez's when $C>4$.

The arguments in the proof of Theorem \ref{conj} extend to Riemannian foliations, giving the following bound.

\begin{cor}\label{cor:foliation focal radius}
Let $\calF$ be a metric foliation with leaves of dimension $k$ in an $n+k$-dimensional compact manifold $M$ with $\sec\geq 1$. Then
\[
k\leq \left(\frac{\pi }{\foc(\calF)}-1\right)(n-1),
\] 
where $\foc(\calF)$ is the focal radius of the foliation.
\end{cor}
The definition of focal radius is included at the end of section \ref{sec:proof theorem A}. 

It is also possible to give bounds on the fiber dimension related to the length of the shortest nontrivial closed geodesic in the base (that exists by a theorem of Fet and Lyusternik \cite{FL}). 

\begin{theorem}\label{thm:closed geodesic bounds}
Let $M^{n+k}, B^n$ be a compact Riemannian manifolds with $\sec\geq 1$, and 
$\pi:M^{n+k}\to B^n$ a Riemannian submersion. Denote by $l_0$ the length of the shortest closed geodesic in $B$. Then 
\[
k\leq \left(\frac{3\pi}{l_0}-1\right)(n-1).
\]
\end{theorem}


The paper is organized as follows: section \ref{sec:abstract jacobi} recalls some results in \cite{Lyt} about abstract Jacobi fields and reformulates some classical theorems in this setting; it ends with a brief summary of Wilking's transverse equation. Section \ref{sec:index bounds} bounds the index of Lagrangian subspaces of Jacobi fields in several situations needed for the proofs of Theorems \ref{conj} and \ref{thm:closed geodesic bounds}; these appear in sections \ref{sec:proof theorem A} and \ref{sec:Proof theorem B}.

The authors would like to thank F. Galaz-Garc\'{\i}a for helpful comments, and P. Piccione and W. Ziller for letting us know about references \cite{He} and \cite{BTZ} respectively.

\section{The Jacobi equation}\label{sec:abstract jacobi}
\subsection{Jacobi fields in an abstract setting}\label{abstract jacobi}
This section collects a few facts on Jacobi fields from \cite{Lyt}. 
Let $E$ be a euclidean vector space of dimension $m$ with  positive definite inner product $\ip{\,}{\,}$. For a smooth one-parameter family of self adjoint linear maps $R:\R\to \Sym(E)$, we consider the  equation $J''(t)+R(t)J(t)=0$ whose solutions we refer to as \emph{$R$-Jacobi fields} (or just \emph{Jacobi fields} if it is clear from the context to what $R$ we refer). 
We denote by $\Jac^R$ the space of Jacobi fields, a vector space of dimension $2m$; $	\Jac^R$ is a symplectic  vector space with form
$$
\ome:\Jac^R\times\Jac^R\to\R, \quad \ome(X,Y)=\ip{X}{Y'}-\ip{X'}{Y}
$$
where the right hand side of $\omega$ is independent of the $t$ chosen. 
A subspace $W$ is called \emph{isotropic} when $\ome$ vanishes in $W$; a maximal isotropic subspace is called a \emph{Lagrangian subspace}, or simply, a Lagrangian. Since $\ome$ is nondegenerate, it is clear that Lagrangian subspaces are just isotropic subspaces of dimension $m$; in the literature, Lagrangian spaces have often been called \emph{maximal self-adjoint spaces for the Jacobi operator} (see for instance \cite{VerZil} and \cite{Wil}). 

Since the inner product of $E$ is positive definite, zeros of Jacobi fields are isolated; we should mention that this is not true in the case of nonzero signature, as was noticed in \cite{He} and further studied in \cite{PT}.
Therefore, if  $I\subset \R$ is an interval, we can define the \emph{index} of an isotropic subspace $W\subset\Jac^R$ in $I$ as the number of times (with multiplicity) that fields in $W$ vanish in $I$; a more precise definition appears in \cite{Lyt}. We will denote this index as $\ind_W I$.

The indexes of different Lagrangians along the same interval are related by the following inequality in \cite{Lyt}.

\begin{prop}
\label{Lytchak inequality}
Let $E, R, {\Jac}^{R}$ be as previously described. Then for any Lagrangians $L_{1},L_{2}\subset {\Jac}^{R}$ and any interval $I\subset\R$, we have
\begin{equation}\label{desigualdad de lagrangianos}
\left| {\ind}_{L_{1}}I-{\ind}_{L_{2}}I\right|\leq\dim{E}-\dim{\left(L_{1}\cap L_{2}\right)}
\end{equation}
\end{prop}

\medskip

\subsection{The transverse Jacobi equation}
Let $W$ be an isotropic subspace of Jacobi fields of $(E, R)$ and $t\in\R$, we define
$$
W(t)=\left\{J(t)\,:\, J\in W\right\}, \quad W^t=\left\{J\in W\,:\,J(t)=0\right\}.
$$
For each $t\in \R$, the subspace
$$
\bW(t)=W(t)\oplus\left\{J'(t)\,:\,J\in W^t\right\}
$$ 
varies smoothly on $t$ as was shown in \cite{Wil}; denote by $H(t)$ its orthogonal complement, and by $e=e^h+e^v$ the splitting of a vector under the sum $E=H(t)\oplus \bW(t)$. We use $\calH$ to denote the vector bundle over $\R$ formed by the $H(t)$. There is a covariant derivative on $\calH$ induced from $E$: if $X:\R\to E$ is a section of $\calH$, we define
$$
\frac{D^hX}{dt}(t)= X'(t)^h.
$$
The covariant derivative $D^h/dt$ defines parallel sections, and preserves the inner product induced on $\calH$ from $E$. Let $E_1$ be an inner vector space of dimension the rank of $\calH$; using a parallel trivialization of $\calH$, we can identify sections of $\calH$ with maps $X:\R\to E_1$, and the covariant derivative $D^h/dt$ with standard derivation. 

 Modulo these identifications, Wilking's transverse equation reads as
 $$
 X''(t)+R^W(t)X(t)=0, \quad R^W(t)X(t)=\left[R(t)X(t)\right]^h+3 A_tA_t^*X(t),
 $$
 where for each $t$, the map $A_t:W(t)\to H(t)$ is linear and its definition can be found in \cite{Wil}. 

Thus we obtain a new Jacobi setting $(E_1, R^W)$ with $R^W$ as the new curvature operator used to construct the Jacobi equation; Wilking proved that the projection of any $R$-Jacobi field onto $\calH$ is a solution of the transverse equation, i.e. an $R^W$-Jacobi field. Moreover, as Lytchak observed,  any Lagrangian for $(E_1, R^W)$ is obtained projecting some Lagrangian that contains $W$ and vice versa.

\section{Bounds on the index}\label{sec:index bounds}
Recall that given an interval $I\subset\R$ and a Lagrangian $L\subset \Jac^R$, we denote by $\ind_L I$ the index of $L$ in $I$ and by $\ind_L(t_0)$ the dimension of the vector subspace of $L$ formed by those Jacobi fields in $L$ that vanish at $t_0$.

Given $a\in\R$, denote by 
$L_a$ the Lagrangian subspace of $\Jac^R$ defined as
\begin{equation}\label{eq:los L_a}
L_a := \left\{\, Y\in\Jac^R\,:\, Y(a)=0\,\right\}.
\end{equation}
Recall that $\dim E=m$.
\subsection{Upper bounds based on the conjugate radius}
All along this subsection we will assume that there is some positive number $c>0$ such that for any $a\in\R$ and any Jacobi field with $Y(a)=0$, $Y$ does not vanish again in $(a,a+c]$. Clearly
$$
\ind_{L_a}(a, a+c)=\ind_{L_a}(a, a+c] =0, \quad \ind_{L_a}[a, a+c) =m.
$$
Inequality \eqref{desigualdad de lagrangianos} shows that for an arbitrary Lagrangian $L$, 
\begin{equation}\label{eq:index arbitrary}
\ind_L(a,a+c]\leq m.
\end{equation}

Our next aim is to improve this to  larger intervals:
\begin{prop}\label{upper bound for Lagrangians - conj-radius}
For any Lagrangian $L$ and any positive integer $r$ we have
$$
\ind_L[a,a+rc]\leq (r+1)m.
$$
\end{prop}

\begin{proof}
Breaking the interval $[a,a+rc]$ into subintervals of length $c$ and using \eqref{eq:index arbitrary} repeatedly, we get that
$$
\ind_{L}[a,a+rc]=\ind_{L}(0)+\sum_{i=0}^{r-1}\ind_{L}(a+ic,a+(i+1)c] 
\leq m+rm.
$$
\end{proof}

\subsection{Curvature-related lower bounds}
To get a lower bound on the index of a Lagrangian $L$, we need to establish the existence of conjugate points for the fields in $L$; Rauch's theorem gives precisely that for a Lagrangian of the form $L_a$ as defined in \eqref{eq:los L_a}. We will then use Proposition \ref{Lytchak inequality} to relate this to the index of an arbitrary Lagrangian. 
 
 We will say that  the curvature $R$ satisfies $R(t)\geq \delta$ for all $t\in\R$ if 
$\ip{R(t)v}{v}\geq \delta\|{v}\|^2$ for any vector $v\in E$. Our first result is a quantitative refinement of Corollary 10 in \cite{Wil}. 

\begin{prop}
\label{better rauch}
Assume that there is some $\delta>0$ such that the curvature $R$ satisfies $R(t)\geq \delta$ for all $t\in\R$. Then for any $a\in\R$, the set 
$$
\calA=\left\{\,Y\in L_a\,:\,Y(t)=0 \text{ for some } t\in \left(a,a+\didelta\right]   \,\right\}
$$
generates $L_a$.
\end{prop}
\begin{proof}
The proof consists on  using of Wilking's transversal equation repeatedly. We describe how to proceed:
\begin{enumerate}
\item We compare $R(t)$ to the constant curvature case $\bR(t)= \delta I$; Rauch's theorem gives us that there is some nonzero $Y_1$ in $\calA$ vanishing for some $t_1\in (a,a+\didelta]$.
\item Let $W_1\subset L_a$ be the vector subspace generated by $Y_1$; we consider the transverse Jacobi equation induced by $W_1$ in $L_a$. In $L_a/W_1$ there is a Jacobi equation of the form 
$$
Y''+R_1Y=0, \quad R_1(t)=R(t)^h+3A_tA_t^*,
$$
and therefore $\ip{R_1(t)v}{v}\geq\ip{R(t)v}{v}\geq \delta\|v\|^2$ for any $v\in W_1(t)^\perp\subset E$. 
Moreover, after taking the $W_1$-orthogonal component,  the fields in $L_a$ give an $R_1$-Lagrangian $L_1$. It is clear that every vector field in $L_1$ vanishes at $t=a$. 
\item Once again, we compare $R_1(t)$ to $\delta I$ to obtain some nonzero $X_2\in L_a$ such that $X_2^\perp$ vanishes at some time $t_2$ in $(a,a+\didelta]$; this merely means that 
$
X_2(t_2)=\lambda Y_1(t_2)
$ for some $\lambda\in\R$, and thus the field
$Y_2=X_2-\lambda Y_1$ is linearly independent with respect to $Y_1$ and lies in $\calA$.
\item Clearly, the process can be iterated as needed until we obtain a basis of $L_a$.
\end{enumerate} 
\end{proof}

Proposition \ref{better rauch} allows us to obtain good lower bounds for the index of a Lagrangian over long intervals. They can also be obtained using the Morse-Schoenberg lemma \cite{Sak} and Proposition \ref{Lytchak inequality}.

\begin{prop}\label{lower Lagrangian bound in long intervals}
Let $a\in\R$; when $R\geq\delta$, the index of any Lagrangian subspace $L$ of Jacobi fields satisfies 
\[
\ind_L\left[a,a+r\pi/\sqrt{\delta}\right]\geq rm+\ind_L(a)
\]
for any positive integer $r$.
\end{prop}
\begin{proof} Without loss of generality we can assume that $a=0$ and write the proof for this case. 
Consider the closed intervals
\[
I_j=[j\didelta,(j+1)\didelta].
\] 
Proposition \ref{better rauch} says that 
\[
\ind_{L_{j\didelta}}I_j\geq 2m;
\]
while Proposition \ref{Lytchak inequality} gives us 
\[
\ind_LI_j\geq \ind_{L_{j\didelta}}I_j-m+\dim(L\cap L_{j\didelta})\geq m+\ind_L(j\didelta).
\]
Breaking the interval $[0,r\didelta]$ into the $I_j$'s, we conclude that
\begin{multline}
\ind_L\left[0,r\Didelta\right] =
\sum_{j=0}^{r-1}\ind_L I_j-\sum_{j=1}^{r-1}\ind_L\left(j\Didelta\right)\geq \\
\geq 
\sum_{j=0}^{r-1}\left(m+\ind_L\left(j\Didelta\right)\right)-\sum_{j=1}^{r-1}\ind_L\left(j\Didelta\right)=
rm+\ind_L(0).
\end{multline}

\end{proof}

A consequence of the last results is the following extension of Proposition \ref{better rauch} to arbitrary Lagrangians:
\begin{prop}\label{prop:span_arbitrary_Lagrangian}
Let $a\in \R$; when $R\geq\delta$, for every Lagrangian subspace $L$ of ${\Jac}^{R}$ the set
\[
\left\{\, Y\in L\,:\,\ Y(t)=0 \text{ for some } t\in \left(a,a+\didelta\right]\, \right\}
\]
spans $L$.
\end{prop}

\begin{proof}
Proposition \ref{lower Lagrangian bound in long intervals} for $r=1$ gives 
\[
\ind_L[a,a+\didelta]\geq m+\ind_L(a).
\]
Therefore there exists a $Y_1\in L$ such that 
$Y_{1}(t_1)=0$ for some $t_1\in (a, a+\didelta]$. The proof is then identical to that of Proposition \ref{better rauch}.
\end{proof}

This has the following geometric application:\begin{thm}\label{thm:submanifold}
Let $M$ be an $n$-dimensional manifold with $\sec\geq 1$ and $\al:\R\to M$ a geodesic orthogonal to a submanifold $N$ at $\al(0)$. Then there are at least $n-1$ focal points of $N$ along $\al$ in the interval $(0,\pi]$.
\end{thm}
\begin{proof}
This follows from Proposition \ref{prop:span_arbitrary_Lagrangian} after choosing the Lagrangian of $N$-Jacobi fields along $\al$ defined as 
\[
L^{N}=
\left\{ 
J\in\Jac^R \,:\, J(0)\in T_{\al(0)}N, J'(0)+S_{\al'(0)}J(0)\perp T_{\al(0)}N 
\right\}
\]
\end{proof}

\subsection{Index bounds for periodic Jacobi fields}

In this section we examine the index of Lagrangians when the solutions of the Jacobi equation are periodic with common period. We will show that such index is always bounded above by some linear function related to multiples of the period.

\begin{prop}\label{prop:upper bound closed geodesic}
Suppose there is some $l>0$ such that for every Jacobi field $J$, the field $t\to J(t+l)$ is also a Jacobi field.
 Then for any Lagrangian $L$ in $\Jac^R$ we have 
\[
\ind_L[a,a+rl] \leq r\left(m+\ind_L[a,a+l)\right)+m
\]
for any positive integer $r$.
\end{prop}

\begin{proof}
As usual, we will write the proof for $a=0$. We start by choosing some basis of $L$, given by $X_1,\dots, X_m$; 
for any positive integer $r$, consider the Jacobi fields defined as 
\[
X_i^r(t)=X_i(t+rl), \quad i=1,\dots, m.
\]
Let $L^r$ the subspace generated by $X_1^r,\dots, X_m^r$; it is Lagrangian, with $L^0=L$. Clearly
\[
\ind_L\left[jl,(j+1)l\right)=\ind_{L^j}[0,l).
\]
 
Using (\ref{desigualdad de lagrangianos}), we have that 
\[
\ind_{L}[0,rl]=\sum_{j=0}^{r-1} \ind_{L^j}[0,l)+\ind_{L^r}(0) 
\leq \sum_{j=0}^{r-1}\left(\ind_L\left[0,l\right)+m\right)+m,
\]
as claimed. 

\end{proof}

%
%
%

\section{Proof of Theorem \ref{conj}}\label{sec:proof theorem A}
We recall some basic facts about Riemannian submersions before proving Theorem \ref{conj}.
Let $\pi:M\to B$ be such a submersion where $M$ and $B$ have dimensions  $n+k$ and $n$ respectively. We will usually overline the notation for objects in the base, to distinguish them from those in $M$. To facilitate the reading, we recall briefly some of the main facts about \emph{projectable Jacobi fields}; the reader can find more information about them in \cite{ON} and \cite[section 1.6]{GrW}; we will use, in particular, the notation from this latter reference.

\begin{defi}
Choose some unit speed geodesic $\al:I\to M$ horizontal for the submersion, and denote by $\bal$ its image $\pi\circ\al$ in $B$. 
A Jacobi field $Y$ along $\al$ is \emph{projectable} if it satisfies
$$
Y'^v=-S_{\al'}Y^v-A_{\al'}Y^h
$$
where $S_{\al'}$ and $A_{\al'}$ are the second fundamental forms of the fibers and the O'Neill tensor of the submersion respectively.
\end{defi}  

The interest of projectable Jacobi fields is that they arise from variations by horizontal geodesics. As such, if $Y$ is a projectable Jacobi field, $\pi_*Y$ is a Jacobi field along $\bal$ in the base. Conversely, we have the following 

\begin{lem}
Let $\bY$ be a Jacobi field of $B$ along $\bal$, and $v$ a vertical vector at $\al(0)$; then there is a unique projectable Jacobi field $Y$ along $\al$ such that 
$\pi_*Y=\bY$ and $Y(0)^v=v$.
\end{lem}  
  
  A particular case of projectable Jacobi fields arises from taking geodesic variations obtained from lifting a geodesic in the base; such fields are called \emph{holonomy Jacobi fields}, and they satisfy the stronger condition
  \[
  J'=-A_{\al'}^*J-S_{\al'}J
  \]
It is clear that they agree with those projectable Jacobi fields mapping to the zero field under $\pi_*$.

  \begin{proof}[Proof of Theorem \ref{conj}]
 Let $\al:\R\to M$ be a horizontal geodesic and $\bal=\pi\circ \al$ its projection by the submersion. The Lagrangian subspace $\bL_o$ can be lifted to $\al$ by considering the subspace spanned by projectable Jacobi fields $Y$ that vanish at $t=0$ (and will therefore have horizontal $Y'(0)$), and by holonomy Jacobi fields along $\al$. We use $L$ to denote such Lagrangian, and $W$ to denote the subspace generated by the holonomy fields. It is interesting to observe that $L$ agrees with the $L^N$ from the proof of Theorem \ref{thm:submanifold} when $N$ is the fiber through $\al(0)$.
 By Lemma 3.1 in \cite{Lyt}, 
\begin{equation}\label{eq:sum of indexes}
 \ind_{L/W}+\ind_W=\ind_L
\end{equation} 
along any interval, where in $L/W$ we are using the transverse Jacobi equation induced by $W\subset L$. Observe that since holonomy Jacobi fields never vanish, $\ind_W=0$ over any interval. We claim that 
$  \ind_{L/W}=\ind_{\bL_0}$. To prove it, we use that, as stated in \cite[section 3.2]{Lyt}, the transverse Jacobi equation corresponding to $W$ along $\al$ agrees with the usual Jacobi equation along $\bal$. Since Lagrangians for the Jacobi equation project to Lagrangians for the transverse Jacobi equation, and every field $Y$ in $L$ satisfies $Y(0)\in W(0)$, we have the mentioned equivalence of indices. Thus we have
$\ind_{\bL_0}=\ind_L$.

 We will estimate this common value over the intervals $[0,r\pi]$ using some of the previous inequalities on the index; choose an arbitrary $c<c_0$ where $c_0$ is the conjugate radius of $B$:
 
 {\setlength\arraycolsep{0.3em}
 \begin{eqnarray}
 \ind_{\bL_0}[0,r\pi] &=& \ind_{\bL_0}\left[0,\frac{r\pi}{c}\cdot c\right]\leq \left(\left[\frac{r\pi}{c}\right]+1\right)(n-1) \\
 \ind_L[0,r\pi] &\geq& r(n-1+k)+\ind_L(0)=r(n-1+k)+(n-1)
 \end{eqnarray}
 }
by propositions 
\ref{upper bound for Lagrangians - conj-radius}
and \ref{lower Lagrangian bound in long intervals} respectively.

	To finish the proof, divide both inequalities by $r$ and make it tend to infinity to conclude that 
	\[
	k\leq \left(\frac{\pi}{c}-1\right)(n-1).
	\]
Letting $c$ tend to $c_0$ gives us Theorem \ref{conj}. 
  \end{proof}
  
The above proof can be easily extended to metric foliations. In order to do this, we define the \emph{focal radius of a metric foliation $\calF$}, $\foc(\calF)$, as the infimum over all the leaves of $\calF$ of the focal radius of each leaf. i.e, the minimum distance to $N$ at which its first focal point appears.

\begin{proof}[Proof of Corollary \ref{cor:foliation focal radius}]
Let $F$ be a leaf of $\calF$ and $\al:\R\to M$ a geodesic orthogonal to $F$ with $\al(0)\in F$. Denote by $W$ the set of holonomy Jacobi fields along $\al$, and by $L$ the Lagrangian spanned by $W$ and those Jacobi fields along $\al$ with $J(0)=0$, $J'(0)\perp T_{\al(0)}F$. Since $\ind_W I=0$, equation \ref{eq:sum of indexes} gives 
\[
\ind_L I=\ind_{L/W} I
\]
for any interval $I$. Observe that $L/W$ corresponds to the Lagrangian $\bL_0=\{ J\,:\, J(0)=0\,\}$ of Jacobi fields for Wilking's transverse equation for the isotropic $W$.

From the definition of the focal radius of $\calF$ it follows that for every $c<\foc(\calF)$, 
\[
\ind_{\bL_0}(0,c]=0,
\]
and therefore we are in the situation of Proposition \ref{prop:span_arbitrary_Lagrangian}, thus
 \[
 r(n-1+k)\leq \ind_L[0,r\pi]=\ind_{\bL_0}[0,r\pi]\leq
 \left(\left[\frac{r\pi}{c}\right]+1\right)\left(n-1\right)
 \]
 for any integer $r>0$. As before, divide both sides by $r$ and let it tend to zero to obtain the inequality claimed in the corollary.
\end{proof}  
  
\section{Proof of Theorem \ref{thm:closed geodesic bounds}} \label{sec:Proof theorem B}

Let $m$ be the smallest positive integer with $\pi_iB=0$ when $i=1,\dots,m-1$ and $\pi_mB\neq 0$. Hurewicz's theorem implies that $m\leq n$. If $\Lambda B$ denotes the free loop space of $B$, then 
$\pi_{m-1}\Lambda B=\pi_mB$, and Lyusternik-Schnirelmann theory implies that there is a closed geodesic  $\bal:[0,\ell]\to B$ such that the number of conjugate points to $\bal(0)$ along $\bal$ in the interval $(0,\ell)$ does not exceed $m-1$ (see \cite[Theorem 1.3]{BTZ}).   
Denote by $\al:\R\to M$ some horizontal lift of $\bal$ to $M$. 

Choose along $\al$ the Lagrangian $L$ of Jacobi fields spanned by the vertical holonomy Jacobi fields and projectable Jacobi fields that vanish at $t=0$. 
As in the proof of Theorem \ref{conj}, we have

\begin{equation}\label{eq:transverse_holonomy_ineq_for_index}
\ind_{L}I= \ind_{\bL_0}I.
\end{equation}

We are going to use this equality in intervals of the form $[0,r\pi]$ for $r$ a positive integer; 
the left hand side in \eqref{eq:transverse_holonomy_ineq_for_index} can be bound with  the help of Proposition \ref{lower Lagrangian bound in long intervals}, giving 
\[
r(n-1+k)+\ind_{L}(0)\leq \ind_{L}[0,r\pi];
\]
on the other hand the right hand side can be bound with Proposition \ref{prop:upper bound closed geodesic} to get
\[
\ind_{\bL_0}[0,r\pi]\leq \left(\left[\frac{r\pi}{\ell}\right]+1\right)\left(n-1+\ind_{\bL_0}[0,\ell)\right). 
\]
Dividing by $r$ and letting it tend to infinity gives
\[
n-1+k\leq \frac{\pi}{\ell}\left(n-1+\ind_{\bL_0}[0,\ell)\right). 
\]
But from the bound on the number of conjugate points of $\bal$ in $[0,\ell)$ we get that 
\[
n-1+k\leq \frac{3\pi}{\ell}\left(n-1\right)\leq  
\frac{3\pi}{\ell_0}(n-1),
\]
which proves Theorem \ref{thm:closed geodesic bounds}.
\qed


\begin{thebibliography}{10}
\bibitem{AK} Amman, M. and Kennard, L., \emph{Positive curvature and rational ellipticity}, preprint, 
\texttt{http://arxiv.org/abs/1403.1440}

\bibitem{BTZ} Ballmann, W.; Thorbergsson, G.; Ziller, W., \emph{Some existence theorems for closed geodesics}, Comment. Math. Helv. 58 (1983), no. 3, 416--432.

\bibitem{dC} do Carmo, M.,\emph{Riemannian geometry},  Mathematics: Theory and Applications. Birkh\"{a}user Boston, Inc., Boston, MA, (1992).

\bibitem{FL} Fet, A. I., Lyusternik, L. A.,  \emph{Variational problems on closed manifolds}, (Russian) Doklady Akad. Nauk SSSR (N.S.) 81, (1951). 17--18.

\bibitem{GrW} Gromoll, D., Walschap, G., \emph{Metric Foliations and Curvature}, Progr. in Math., 268, Birkhauser Verlag (2009).

\bibitem{He} Helfer, A. D., \emph{Conjugate points on spacelike geodesics or pseudo-selfadjoint Morse-Sturm-Liouville systems}, Pacific J. Math. 164 (1994), no. 2, 321--350.

\bibitem{BG} Jim\'enez, W., \emph{Riemannian Submersions And Lie Groups}, thesis UMD (2006), available at
\texttt{http://hdl.handle.net/1903/2648}. 

\bibitem{KT} Kim, H. and Tondeur, P., \emph{Riemannian foliations on manifolds with nonnegative curvature}, 
Manuscripta Math. 74 (1992), no. 1, 39--45. 

\bibitem{Lyt} Lytchak, A., \emph{Notes on the Jacobi equation}, Differential Geom. Appl. 27 (2009), no. 2, 329--334. 

\bibitem{ON} O'Neill, B., \emph{Submersions and geodesics}, Duke Math. J. 34 (1967), 363--373.

\bibitem{PT} Piccione, P.; Tausk, D.,
\emph{On the distribution of conjugate points along semi-Riemannian geodesics},
Comm. Anal. Geom. 11 (2003), no. 1, 33--48. 

\bibitem{Sak} Sakai, T., \emph{Riemannian geometry}, Translations of Mathematical Monographs, 149. American Mathematical Society, Providence, RI, (1996).


\bibitem{VerZil} Verdiani, L., Ziller, W., \emph{Concavity and rigidity in non-negative curvature}, J. Diff. Geom. 97 (2014), 349--375.

\bibitem{Wals} 
Walschap, G.,  \emph{Metric foliations and curvature},  J. Geom. Anal. 2 (1992), no. 4, 373--381. 

\bibitem{Wil} Wilking, Burkhard; \emph{A duality theorem for Riemannian foliations in nonnegative sectional curvature}, Geom. Funct. Anal. 17 (2007), no. 4, 1297--1320.






\end{thebibliography}
\end{document}